\documentclass[12pt]{article}

\usepackage{graphicx}
\usepackage{amssymb,amsmath,amsthm}
\usepackage{tikz}

\newtheorem{lemma}{Lemma}
\newtheorem{theorem}{Theorem}

\newcommand{\ds}{\displaystyle}

\advance \oddsidemargin -0.45in
\advance \textwidth 1.0in
\advance \textheight 1.75in
\advance \topmargin -.75in

\title{On the Convergence of the Variational Iteration Method for Klein-Gordon Problems with Variable Coefficients II}

\author{
{\normalsize{\sc Pavel Dr\'{a}bek}}
\thanks{ Research sponsored by the Grant Agency of the Czech Republic, grant no. 22-182615} \\
{\scriptsize Department of Mathematics and Center N.T.I.S., University of West Bohemia,
P.O. Box 314, 306 14 Pilsen}\\[-1ex]
{\scriptsize e-mail:pdrabek@kma.zcu.cz}
\and
{\normalsize{\sc Stephen B. Robinson}}\\
{\scriptsize Department of Mathematics, Wake Forest University, Winston-Salem, NC
27109}\\[-1ex]
{\scriptsize e-mail: sbr@wfu.edu}
\and
{\normalsize{\sc Shohreh Gholizadeh Siahmazgi}}\\
{\scriptsize Department of Mathematics, Wake Forest University, Winston-Salem, NC
27109}\\[-1ex]
{\scriptsize e-mail: ghols18@wfu.edu}
}

\date{ }

\begin{document}

\maketitle

\begin{abstract}
In this paper we investigate convergence for the Variational Iteration Method (VIM) which was introduced and described in \cite{He0},\cite{He1}, \cite{He2}, and \cite{He3}. We prove the convergence of the iteration scheme for a linear Klein-Gorden equation with a variable coefficient whose unique solution is known. The iteration scheme depends on a {\em Lagrange multiplier}, $\lambda(r,s)$, which is represented as a power series. We show that the VIM iteration scheme converges uniformly on compact intervals to the unique solution. We also prove convergence when $\lambda(r,s)$ is replaced by any of its partial sums. The first proof follows a familiar pattern, but the second requires a new approach. The second approach also provides some detail regarding the structure of the iterates.
\end{abstract}

\noindent Keywords: Variational Iteration Method, Klein-Gordon, Convergence

\noindent MSC: 35A,35L,35Q

\section{Introduction}

In this paper we investigate convergence for the Variational Iteration Method (VIM) which was introduced and described in \cite{He0},\cite{He1}, \cite{He2}, and \cite{He3}. We prove the convergence of the iteration scheme for a linear Klein-Gorden equation with a variable coefficient whose unique solution is known. This equation serves as a test case. The iteration scheme depends on a {\em Lagrange multiplier}, $\lambda(r,s)$, which is represented as a power series. We show that the VIM iteration scheme converges uniformly on compact intervals to the unique solution. We also prove convergence when $\lambda(r,s)$ is replaced by any of its partial sums. The first proof follows a familiar pattern, but the second requires a new approach. The second approach also provides some detail regarding the structure of the iterates.

Consider the linear Klein-Gordon problem
\begin{equation}
\begin{array}{l}
u_{rr}-u_{tt}+ru=0, r>0,-\infty<t<\infty\\
u(0,t)=e^{it},\\
u_r(0,t)=0.
\end{array}
\label{IVP}
\end{equation}
The unique solution can be found using elementary separation of variables and power series methods. It is
\begin{equation}
u(r,t)=e^{i t}\sum_{k=0}^{\infty}a_kr^k,
\label{solution}
\end{equation}
where the sum is an Airy function with $a_0=1,a_1=0,a_2=-\frac{1}{2}$, and
\begin{equation}
a_k=-\left ( \frac{a_{k-3}+a_{k-2}}{k(k-1)}\right ), k\geq 3.
\label{recursionrelation1}
\end{equation}
We will refer to these as {\em Airy coefficients} in discussions below.

The {\em Lagrange multiplier} is given by
\begin{equation}\label{lambda}
\lambda(r,s)=\sum_{k=0}^{\infty}\alpha_k (s-r)^k,
\end{equation}
where $\alpha_0=0,\alpha_1=1,\alpha_2=0$, and
\begin{equation}\label{recursionrelation2}
\alpha_k=-\frac{\alpha_{k-3}+r\alpha_{k-2}}{k(k-1)}, k\geq 3.
\end{equation}
Observe that this recursion relation is similar to \eqref{recursionrelation1}, but is not the same. We denote the partial sums as
\begin{equation}\label{lambdaN}
\lambda_N(r,s)=\sum_{k=0}^{N}\alpha_k (s-r)^k.
\end{equation}
As will be seen below $\lambda(r,s)$ can be derived via the optimization of an appropriate Lagrangian, or as a Green's Function for an initial value problem.

The standard VIM leads to an iteration scheme where $u_0(r,t)=e^{it}$  and
\begin{equation}\label{lambdaiteration}
u_{n+1}(r,t)=u_n(r,t)+\int_{0}^{r}\lambda (r,s)\left (u_{n,ss}(s,t)-u_{n,tt}(s,t)+su_{n}(s,t) \right )ds.
\end{equation}
We will also investigate the same scheme with $\lambda_N(r,s)$ replacing $\lambda(r,s)$, i.e.
\begin{equation}\label{lambdaNiteration}
u_{n+1}(r,t)=u_n(r,t)+\int_{0}^{r}\lambda_N (r,s)\left (u_{n,ss}(s,t)-u_{n,tt}(s,t)+su_{n}(s,t) \right )ds.
\end{equation}

We will prove the following theorems.
\begin{theorem}\label{convergence1}
Given any $R>0$ the sequence of iterates given by $u_0(r,t)=e^{it}$ and \eqref{lambdaiteration} converges uniformly to the solution \eqref{solution}  for $(r,t)\in [-R,R]\times\mathbb{R}$.
\end{theorem}
\begin{theorem}\label{convergence2}
Given any $N\in\mathbb{N}$ and any $R>0$ the sequence of iterates given by $u_0(r,t)=e^{it}$ and \eqref{lambdaNiteration} converges uniformly to the solution \eqref{solution} for $(r,t)\in [-R,R]\times\mathbb{R}$.
\end{theorem}

Much of the literature on the VIM provides interesting applications with compelling computational evidence for fast convergence. Rigorous proof is less common. However, papers such as \cite{KS}, \cite{O}, \cite{Sa}, \cite{ST}, \cite{S}, \cite{TD},\cite{TK}, and \cite{YX} provide detailed proofs of convergence for a wide variety of problems. A common theme in the proofs are estimates similar to those used in the Contraction Mapping Theorem or Picard Iteration. Our proof of Theorem \ref{convergence1} follows that pattern. This is not especially new, but we provide the details for the purpose of comparison. It is interesting that this standard approach to proof does not adapt well to the iteration using partial sums of $\lambda(r,s)$. Our proof of Theorem \ref{convergence2} derives a recursion relation for the coefficients of the iterates followed by an application of comparison methods. We significantly generalize the work in \cite{Si} and \cite{SiR}, where several special cases were considered.

Most of the literature dealing with the application of the VIM to initial value problems for wavelike equations, such as Klein-Gordon, assume constant coefficients. See \cite{H}, \cite{He3}, \cite{K},\cite{TD},or \cite{W} for examples including both linear and nonlinear problems on finite and infinite intervals. These papers also rely on the simplest available choice of $\lambda(r,s)$ associated with the problem. In the following section we will see that VIM allows some flexibility in the choice of $\lambda(r,s)$.

In section 2 we briefly review several perspectives on the formal derivation of the VIM and its associated Lagrange multiplier. Section 3 provides  several useful formulae describing $\lambda(r,s)$ and the coefficients of the VIM iterates for \eqref{lambdaNiteration}. Section 4 provides the proofs of Theorems \ref{convergence1} and \ref{convergence2}. Section 5 suggests a few directions for further research.

\section{Formal derivation of the VIM}

In this section we briefly review ideas described in \cite{ISM}, \cite{He0}, and \cite{SiR}. The first paper describes a generalized Newton's method that is credited by several authors as motivation for the VIM.  The second provides a common derivation of the VIM iteration scheme which modifies the generalized Newton's method by introducing the idea of a {\em restricted variation}. Further examples and explanation of the VIM can be found in \cite{He1}, \cite{He2}, and \cite{He3}. The third provides an alternative derivation  of the VIM where the Lagrange multiplier $\lambda(r,s)$ is viewed as a Green's function. We emphasize that these derivations are formal and do not provide details of the relevant function spaces, norms, etc.

\subsection{A generalized Newton's Method}

Represent a differential equation and its auxiliary condition(s) abstractly as
\[
F[u]=0.
\]
Assume that we are looking for solutions $u$ in a Banach space with duality pairing $\langle \cdot,\cdot\rangle$.
Represent a quantity associated with the solution as $q=Q[u]$. ( One choice is the solution itself, i.e. $Q[u]=u$, but the discussion in \cite{ISM} allows greater generality.)
Let $u_0$ represent a first approximation of the solution and let $u$ represent the actual solution which is unknown. We want to estimate
\[
q=Q[u].
\]
We have $q_0=Q[u_0]$ but want an improved estimate $q_1$.
Write $u_0=u+\delta u$, where $\delta u$ represents a variation or test function,  and linearize $q=Q[u]-\langle \lambda_0, F[u]\rangle$ at $u$, so that
\[
\begin{array}{lll}
q&\approx & Q[u]+\langle DQ[u],\delta u\rangle -\langle \lambda_0,DF[u]\cdot\delta u\rangle \\
&= &  Q[u]+\langle DQ[u],\delta u\rangle -\langle \left (DF[u] \right )^* \cdot\lambda_0,\delta u\rangle \\
&= &  Q[u]+\langle DQ[u] -\left (DF[u] \right )^* \cdot\lambda_0,\delta u\rangle,
\end{array}
\]
where the derivatives represent Frechet derivatives, and $\left ( DF[u] \right )^*$ is the adjoint operator associated with $DF[u]$.
The {\em optimal} $\lambda_0$ will satisfy the linear equation
\[
DQ[u] -\left ( DF[u] \right )^*\lambda_0=0.
\]
Since $u$ is unknown we solve for $\lambda_0$ using
\[
DQ[u_0] -\left ( DF[u_0] \right )^*\lambda_0=0.
\]
This leads to
\[
q_1=Q[u_0]-\langle \lambda_0, F[u_0]\rangle=q_0-\langle \lambda_0, F[u_0]\rangle.
\]

In \cite{ISM} the process above described a single step towards an improved estimate of $q$, but this process can clearly be adapted to an iteration scheme. Each additional approximation step would require an updated $n$th approximate solution, $u_n$, and an updated $\lambda_n$.

\subsection{The standard derivation of the VIM}

In \cite{He0} problems of the following type are considered.
\[
\begin{array}{c}
Lu + N(u)=0,\\
A(u)=g
\end{array}
\]
where $L$ is a linear differential operator in one variable, $N$ may include differential, linear, and/or nonlinear terms, $A(u)=g$ represents a linear auxiliary condition, and $u$ is a smooth classical solution. For \eqref{IVP} we assume $u(r,t)$ is analytic, $Lu=u_{rr}+ru$, $N(u)=-u_{tt}$, $A(u)=(u(0,t),u_r(0,t))$, and $g=(e^{it},0)$.

Let $u_0(r,t)$, satisfying $A(u_0)=g$, be the initial iterate and let
\[
u_1=u_0+\int_{0}^{r}\lambda (r,s)\left ( Lu_0+N(u_0)\right )ds.
\]
It is straight forward to verify that $A(u_1)=g$. Similarly to the previous section we now want to determine an optimal $\lambda$.
Let $u_0=u_1+\delta u$, where $A(\delta u)=0$, and apply a {\em restricted variation}  of $N$, i.e.
\[
u_1=u_1+\delta u+\int_{0}^{r}\lambda (r,s)\left ( L(u_1+\delta u)+N(u_1)\right )ds
\]
It follows that
\[
0=\delta u+\int_{0}^{r}\lambda \left ( L\delta u\right)ds+\int_{0}^t\lambda \left (L(u_1)+N(u_1)\right )ds.
\]
We want $Lu_1+N(u_1)\approx 0$, so the {\em optimal } choice for $\lambda$ satisfies
\begin{equation}\label{InverseEquality}
0=\delta u+\int_{0}^{r}\lambda \left (L \delta u\right)ds
\end{equation}
for all variations $\delta u$ satisfying $A(\delta u)=0$. Using judicious choices of $\delta u$ and integration by parts, we can derive a linear ODE for $\lambda$.  We solve for $\lambda$ and then iterate using
\[
u_{n+1}=u_n+\int_{0}^{r}\lambda (r,s)\left ( Lu_n+N(u_n)\right )ds.
\]

Two observations are interesting here. First, there may be multiple options for the choice of $L$, and subsequently $N$, in a single problem. In our problem we could choose $Lu=u_{rr}, Lu=u_{rr}+ru,$ or $Lu=-u_{tt}$. In the literature the first choice is the most common, leads to a simple $\lambda(r,s)$, and produces good computational results for many examples. Second, $\lambda(r,s)$ does not need to be updated with each iteration as would be necessary with the generalized Newton's method.

\subsection{An alternative derivation of the VIM}

Start with the same problem statement as in the previous section. Choose $u_0$ such that $A(u_0)=g$. Solve for $u_1=u_0+\delta u$ such that
  \[
  \begin{array}{c}
    Lu_1+N(u_0)=0, \\
    A(u_1)=g
    \end{array}
    \]
which can be restated as
    \[
  \begin{array}{c}
    L(\delta u)+(L(u_0)+N(u_0))=0, \\
    A(\delta u)=0
    \end{array}
    \]
Let $L^{-1}$ be the solution operator for
  \[
  \begin{array}{c}
    Lu+f=0, \\
    A(u)=0.
  \end{array}
  \]
Thus $\delta u=-L^{-1}\left (Lu_0+N(u_0) \right )$, and
$u_1=u_0-L^{-1}\left (Lu_0+N(u_0) \right ).$

The formal argument above leaves us with the task of solving for $L^{-1}$. A well-understood approach is to find a Green's function $\lambda(r,s)$ such that
\[
-(L^{-1}f)(r)=\int_{0}^{r}\lambda(r,s)\, f(s)\,ds.
\]
If $\delta u$ is a test function satisfying $A(\delta u)=0$, then $\delta u=L^{-1}L\delta u$ so
\[
\delta u=-\int_{0}^{r}\lambda(r,s)\,L\delta u \,ds.
\]
One can use well chosen test functions and integration-by-parts to derive an initial value problem for an ODE which can then be solved for $\lambda (r,s)$. Once the Green's function $\lambda(r,s)$ has been determined, then we have the iteration scheme
\[
u_{n+1}=u_n-L^{-1}\left ( Lu_n+N(u_n)\right )=u_n+\int_{0}^{r}\lambda(r,s)\left (Lu_n+N(u_n) \right )ds.
\]

\section{Applying the VIM to the test problem}

In this section we apply the VIM to our test problem \eqref{IVP} to produce a sequence of approximate solutions.
Let $u_0=e^{it}$, which satisfies the initial conditions, and let
\[
u_{n+1}=u_n+\int_0^r \lambda(r,s)\left (u_{n,ss}-u_{n,tt}+su_n \right )ds.
\]
For convenience we substitute $u_n=e^{it}\phi_n(r)$ into the iteration scheme and then divide out the exponential to get
\[
\phi_{n+1}=\phi_n+\int_0^r \lambda(r,s)\left (\phi_{n,ss}+s\phi_n+\phi_n \right )ds.
\]

\subsection{Solving for $\lambda$}

We solve for $\lambda$ by considering
\[
\delta u(r)=-\int_{0}^{^r}\lambda (r,s)(\delta u_{ss}(s)+s\delta u(s))ds
\]
for all test functions $\delta u (r)$, i.e. smooth functions $\delta u(r)$ such that $\delta u(0)=1$ and $\delta u_r(0)=0$.
Applying standard integration by parts methods we derive
\[
\begin{array}{c}
  \lambda_{ss}(r,s)+s\lambda(r,s)=0, \\
  \lambda_s(r,r)=1,\mbox{ and}\\
  \lambda(r,r)=0.
\end{array}
\]
This has the classic solution
\[
\begin{array}{lll}
\lambda (r,s)&=&\sum_{k=0}^{\infty}\alpha_{k}(s-r)^{k}\\ \\
&=&(s-r)-\frac{r}{6}(s-r)^{3}-\frac{1}{12}(s-r)^4+\frac{r^2}{120}(s-r)^5+\cdots,
\end{array}
\]
where $\alpha_0=0,\alpha_1=1,\alpha_2=0,$ and
\[
\alpha_k=-\frac{\alpha_{k-3}+r\alpha_{k-2}}{k(k-1)},k\geq 3.
\]
Notice that this is similar to, but not the same as, the recursion formula for the solution. Further observe that the coefficients are polynomials in $r$ where the degree of $\alpha_k$ is no larger than $k$, so we can write
\[
\alpha_k(r)=\sum_{j=0}^{k}\alpha_{kj}r^j,
\]
where $\alpha_{kj}\in\mathbb{R}$.
Let the $N$th partial sum of $\lambda$ be $\lambda_N$.

\subsection{Computing the necessary integrals}

Given the fact that $\phi_n(s)$ will be a power series in $s$, and that $\lambda(r,s)$ is a power series in $(s-r)$ with coefficients that are polynomials in $r$, computing $\phi_{n+1}(r)$ involves repeated computation of a well-known integral expression
\[
\int_{0}^{r}(s-r)^ms^nds=(-1)^m(m!)\frac{r^{m+n+1}}{(n+1)\cdots (m+n+1)}.
\]
For the special case where $r=1$ we get
\[
\int_{0}^{1}(s-1)^ms^nds=(-1)^m\frac{m!}{(n+1)\cdots (m+n+1)}.
\]
This motivates a definition of the {\em Beta function}
\[
\begin{array}{lll}
  B(m,n) & := & \frac{(m-1)!}{(n)\cdots (m+n-1)} \\ \\
   & = & \frac{(m-1)!(n-1)!}{(m+n-1)!}.
\end{array}
\]
One can now verify several useful results such as the formula
\[
\int_{0}^{r}(s-r)^m s^n ds=(-1)^m B(m+1,n+1)r^{m+n+1},
\]
the symmetry property
\[B(n,m)=B(m,n),\] and the recursion formula
\begin{equation}\label{Betarecursion}
B(m,n+1)=B(n,m)\frac{n}{m+n}.
\end{equation}
These formulae and further information about Beta Functions can be found in \cite{D}.

\subsection{Computing iterations using $\lambda_N$ in place of $\lambda$}

Let
\[
\phi_n(r)=\sum_{m=0}^{\infty}a^n_mr^m.
\]
It is convenient to write this as an infinite sum even though the iterates will actually be finite sums.
We substitute into
\[
\phi_{n+1}(r)=\phi_n(r)+\int_{0}^{r}\lambda_N(r,s)\left (\phi_{n,ss}(s)+s\phi_n(s) + \phi_n(s) \right )ds
\]
to get
\[
\phi_{n+1}(r)=\sum_{m=0}^{\infty}a^{n+1}_mr^m.
\]
We can then determine a recursion relation between $(a_m^{n+1})$ and $(a_m^n)$.

A first step is to apply the iteration formula just to $\rho_0=r^m$ to get
\[
\begin{array}{lll}
\rho_1(r)&=&\ds \rho_0(r)+\int_{0}^{r}\lambda_N(s,r)\left (\rho_{0,ss}(s)+s\rho_0(s)+\rho_0(s)\right) ds\\
&=&\ds r^m+\int_{0}^{r}\left (\sum_{k=1}^{N}\alpha_k(s-r)^k\right )\left (m(m-1)s^{m-2}+s^m+s^{m+1}\right )ds\\
&=& \ds r^m+\sum_{k=1}^{N}\alpha_k\int_{0}^{r}(s-r)^k \left ( m(m-1)s^{m-2}+s^m+s^{m+1} \right )ds \\ \\
&& \mbox{using Beta functions this becomes} \\ \\
&=& \ds r^m+\sum_{k=1}^{N}(-1)^k\alpha_k \Big ( m(m-1)B(k+1,m-1)r^{k+m-1}+B(k+1,m+1)r^{k+m+1}\\  && +B(k+1,m+2)r^{k+m+2} \Big )\\ \\
&& \mbox{separating out the first terms and using}\, \alpha_1=1\, \mbox{and}\, \alpha_2=0 \\ \\
&=& \ds r^m -\left (m(m-1)B(2,m-1)r^m+B(2,m+1)r^{m+2}+B(2,m+2)r^{m+3}\right )\\  && \ds +\sum_{k=3}^{N}(-1)^k\alpha_k\Big (m(m-1)B(k+1,m-1)r^{k+m-1}+B(k+1,m+1)r^{k+m+1}
\\ &&+B(k+1,m+2)r^{k+m+2}\Big )\\ \\
&&\mbox{simplifying } m(m-1)B(2,m-1)=1\\ \\
&=& \ds r^m -\left (r^m+B(2,m+1)r^{m+2}+B(2,m+2)r^{m+3}\right )\\  && \ds +\sum_{k=3}^{N}(-1)^k\alpha_k\Big (m(m-1)B(k+1,m-1)r^{k+m-1}+B(k+1,m+1)r^{k+m+1}
\\ &&+B(k+1,m+2)r^{k+m+2} \Big )\\
&=&\ds -\left (B(2,m+1)r^{m+2}+B(2,m+2)r^{m+3}\right )\\  && \ds +\sum_{k=3}^{N}(-1)^k\alpha_k\Big (m(m-1)B(k+1,m-1)r^{k+m-1}+B(k+1,m+1)r^{k+m+1}
\\  &&+B(k+1,m+2)r^{k+m+2}  \Big )\\ \\
&& \mbox{expanding the polynomial coefficients }\alpha_k(r) \\ \\
&=&\ds -\left (B(2,m+1)r^{m+2}+B(2,m+2)r^{m+3}\right )\\  && \ds +\sum_{k=3}^{N}(-1)^k\sum_{j=0}^{k}\alpha_{kj}\Big (m(m-1)B(k+1,m-1)r^{k+m+j-1}+B(k+1,m+1)r^{k+m+j+1}
\\  &&+B(k+1,m+2)r^{k+m+j+2} \Big )
\end{array}
\]
Summarizing, we have
\begin{equation}\label{rm}
    \begin{split}
        \rho_1(r) &=-\left (B(2,m+1)r^{m+2}+B(2,m+2)r^{m+3}\right ) + \\
        &\sum_{k=3}^{N}(-1)^k\sum_{j=0}^{k}\alpha_{kj}\Big ( m(m-1)B(k+1,m-1)r^{k+m+j-1}+B(k+1,m+1)r^{k+m+j+1}
        \\ &\,\,\,\,\,\,\,\,\,\,\,\,\,\,\,\,\,\,\,\,\,\,\,\,\,\,\,\, +B(k+1,m+2)r^{k+m+j+2} \Big )
    \end{split}
\end{equation}

This formula can be used to predict the highest power of $r$ appearing in $(n+1)$st iterate given the highest power in the $n$th iterate. For example, if the highest power in the $n$th iterate is $M$, then the highest power in the $(n+1)$st iterate is at most $M+2N+2$. This uses the fact that $\alpha_N$ may include at most an $N$th power of $r$.

We can also apply the formula above to determine how coefficients from the $(n+1)$st iteration depend on the coefficients of the $n$th iteration. Observe that the term $a_{m*}^n r^{m^*}$ in $\phi_n$ contributes to $a_m^{n+1}r^m$ in $\phi_{n+1}$ if $m^*=m-2,m-3, m-k-j+1,m-k-j-1,$ or $m-k-j-2$. Therefore
\[
\begin{array}{lll}
a_m^{n+1}&=&\ds \ds -\left (B(2,m-1)a^n_{m-2}+B(2,m-1)a^n_{m-3}\right ) \\  &&\ds +\sum_{k=3}^{N}(-1)^k\sum_{j=0}^{k}\alpha_{kj}\Big ((m-j-k+1)(m-j-k)B(k+1,m-j-k)a^n_{m-j-k+1}
\\  && \ds+B(k+1,m-k-j)a^n_{m-j-k-1} +B(k+1,m-j-k)a^n_{m-j-k-2} \Big ) \\ \\
&=&\ds -\left ( a^n_{m-2}+a^n_{m-3}\right )B(2,m-1) \\  &&\ds +\sum_{k=3}^{N}(-1)^k\sum_{j=0}^{k}\alpha_{kj}\Big ((m-j-k+1)(m-j-k)a^n_{m-j-k+1}
\\  && \ds+a^n_{m-j-k-1} +a^n_{m-j-k-2}\Big )B(k+1,m-k-j) \\ \\
&=&\ds -\left (\frac{ a^n_{m-2}+a^n_{m-3}}{m(m-1)}\right )\\  &&\ds +\sum_{k=3}^{N}(-1)^k\sum_{j=0}^{k}\alpha_{kj}\Big (a^n_{m-j-k+1}
\\  && \ds+\frac{a^n_{m-j-k-1} +a^n_{m-j-k-2}}{(m-j-k+1)(m-j-k)} \Big )(m-j-k+1)(m-j-k)B(k+1,m-k-j)
\end{array}
\]
A helpful simplification using \eqref{Betarecursion} is the following
\[
\begin{array}{ll}
  (m-j-k+1)(m-j-k)B(k+1,m-k-j)&  =  \\
  (m-j-k+1)(m-j+1)B(k+1,m-k-j+1)&= \\
   (m-j+2)(m-j+1)B(k+1,m-k-j+2)&=\\
(m-j+1)(m-j+2)\frac{k!}{(m-j+2)\cdots (m-k-j+2)}&= \\
\frac{k!}{(m-j)(m-j-1)\cdots (m-k-j+2)}&
\end{array}
\]
Thus we have
\begin{equation}\label{recursion3}
\begin{array}{l}
  a_m^{n+1}  =  \ds -\left (\frac{ a^n_{m-2}+a^n_{m-3}}{m(m-1)}\right )+\\  \ds \sum_{k=3}^{N}(-1)^k\sum_{j=0}^{k}\alpha_{kj}\Big (a^n_{m-j-k+1}
 \ds+\frac{a^n_{m-j-k-1} +a^n_{m-j-k-2}}{(m-j-k+1)(m-j-k)} \Big )\frac{k!}{(m-j)\cdots (m-k-j+2)}
\end{array}
\end{equation}
Observe that if $a^n_{0},...,a^n_{m-2}$ are Airy coefficients, then
\[
\left (a^n_{m-j-k+1}+\frac{a^n_{m-j-k-1} +a^n_{m-j-k-2}}{(m-j-k+1)(m-j-k)} \right )=0,
\]
for the given ranges of $j$ and $k$, so
\[
a_m^{n+1}  =  -\left (\frac{ a^n_{m-2}+a^n_{m-3}}{m(m-1)}\right ).
\]
Thus
$a^{n+1}_{0},...,a^{n+1}_m$ are Airy coefficients. We collect the above observations in the following lemmas.

\begin{lemma}\label{AiryCoeffs}
If $m\leq 2n+2$, then $(a_m^n)$ is an Airy coefficient.
\end{lemma}

\begin{proof}
We know that the first two coefficients of $\phi_0=1$, i.e. $a_0=1,$ and $a_1=0$, are Airy coefficients. We have shown above that the number of Airy coefficients increases by at least two with each iteration.
\end{proof}

\begin{lemma}\label{degree}
$\phi_n$ is a polynomial of degree at most $(2N+2)n$.
\end{lemma}
\begin{proof}
$\phi_0$ is a polynomial of degree $0$ and we have shown above that the degree of each iterate increases by at most $2N+2$.
\end{proof}

\section{Proof of convergence}

\subsection{Proof of Theorem \ref{convergence1}}

In this section we prove convergence for the standard VIM iteration using $\lambda$, i.e. we assume $\phi_0\equiv 1$ and
\[
\phi_{n+1}(r)=\phi_n(r)+\int_{0}^{r}\lambda(r,s)\left ( \phi_{n,ss}(s)+s\phi_n(s)+\phi_n(s)\right ) ds.
\]
We apply an argument similar to those in \cite{KS}, \cite{O}, \cite{Sa}, \cite{ST}, \cite{S}, \cite{TD},\cite{TK}, and \cite{YX}.
We assume $\phi$ is the exact solution satisfying
\[
\phi (r)=\phi (r)+\int_{0}^{r}\lambda(r,s)\left ( \phi (s)+s\phi (s)+\phi(s)\right ) ds.
\]
Let $e_n(r)=\phi_n(r)-\phi(r)$. Then
\[
\begin{array}{lll}
  e_{n+1}(r) & = & e_n(r)+\int_{0}^{r} \lambda(r,s)\left ( e_{n,ss}(s)+se_n(s)+e_n(s)\right )ds\\ \\
   & = & \int_{0}^{r} \lambda(r,s)e_n(s) ds,
\end{array}
\]
where we have used the fact that since $e_n(0)=0$ and $e_{n,r}(0)=0$ we know
\[
e_n(r)+\int_{0}^{r} \lambda(r,s)\left ( e_{n,ss}(s)+se_n(s)\right )ds=0.
\]
Let $M>0$ such that $|\lambda(s,r)|\leq M$ for all $0\leq s\leq r\leq R$. Then successive integrals yield
\[
\begin{array}{lll}
  |e_1(r)| & \leq  & ||\phi_0-\phi||_{\infty}Mr,\\
  |e_2(r)| & \leq  & ||\phi_0-\phi||_{\infty}\frac{(Mr)^2}{2},\\
   & \hdots &  \\
 |e_n(r)| & \leq  & ||\phi_0-\phi||_{\infty}\frac{(Mr)^n}{n!}.
\end{array}
\]
Uniform convergence follows for $|r|\leq R$.

This approach does not adapt well to a proof of convergence when $\lambda$ is replaced by $\lambda_N$, because it is no longer necessarily the case that
\[
e_n(r)+\int_{0}^{r} \lambda_N(r,s)\left ( e_{n,ss}(s)+se_n(s)\right )ds=0.
\]
It is worth noting that although using $\lambda$ is relatively simple to work with theoretically, it might be computationally expensive relative to using $\lambda_N$.

\subsection{Proof of Theorem \ref{convergence2}}

The argument for convergence is based on the following inequality.
\begin{lemma}\label{inequality}
There is a constant $C>0$ such that for every $m>2N+2$ and every $n\in N$ we have
\[
|a_m^{n+1}|\leq \frac{C}{(m-N)(m-N-1)}\max\{|a^n_{m-2}|,...,|a^n_{m-(2N+2)}|\}.
\]
\end{lemma}

\begin{proof}
  From \eqref{recursion3} we have
  \[
  |a_m^{n+1}|\leq \left (1+2k!\sum_{k=3}^{N}\sum_{j=0}^{k}|\alpha_{kj}|\right )\frac{1}{(m-N)(m-N-1)}\max\{|a^n_{m-2}|,...,|a^n_{m-(2N+2)}|\},
  \]
  where we have used
  \[
  \left (1+\frac{2}{(m-j-k+1)(m-j-k)}\right )\leq 2
  \]
 and
  \[\frac{1}{(m-j)(m-j-1)\cdots (m-k-j+2)}\leq \frac{1}{(m-N)(m-N+1)}\] for all $j,k$ in the given ranges.

\end{proof}

\begin{lemma}
Given any $\mu\in\mathbb{N}$ the set $\{|a_m^n|:m\leq\mu,n\in\mathbb{N}\}$ is bounded
\end{lemma}

\begin{proof}
By Lemma \ref{AiryCoeffs} the set $\{|a_m^n|:m\leq 2n\}$ is a subset of the classic Airy coefficients and is therefore bounded. The set $\{|a_m^n|:2n\leq m\leq\mu\}$ is finite and therefore bounded. The result follows.
\end{proof}

\begin{lemma}
$\{a_n^m:n,m\in \mathbb{N}\cup\{0\}\}$ is bounded.
\end{lemma}

\begin{proof}
Choose $\mu\in\mathbb{N}$ such that $\frac{M}{(m-N)(m-N-1)}<\frac{1}{2}$ for $m\geq\mu$. Let $B$ represent the bound on $\{|a_m^n|:m\leq\mu,n\in\mathbb{N}\}$. Then Lemma \ref{inequality} implies that for all $n$ we have $|a_{\mu+2}^n|\leq B$, and $|a_{\mu+1}^n|\leq B$. Thus we can proceed inductively to show $|a_m^n|\leq B$ for all $n$ and $m$.
\end{proof}

The previous lemmas allow us to iterate our main inequality to get the following lemma. Before stating the result we define
\begin{equation}\label{FunnyFactorial}
  k\tilde{!}:=\begin{cases}
                k, & \mbox{if } k\leq 2N+2 \\
                k(k-1)\cdot (k-(2N+2))\tilde{!}, & \mbox{otherwise}.
              \end{cases}
\end{equation}

\begin{lemma}\label{comp1}
If $m=(2N+2)d +\rho$, where $\rho\in \{1,...,2N+2\}$, then
\[
|a_n^m|\leq B\frac{C^d}{(m-N)\tilde{!}}\leq  B\frac{C^{\frac{m}{2N+2}}}{(m-N)\tilde{!}}.
\]
\end{lemma}

\begin{proof}
For the given $m$ if $a_n^m\neq 0$, then we know the degree of $\phi_n$ is greater than $d(2N+2)$, so $n>d$. Thus we can iterate the inequality $d$ times as follows. Note that at each step we are applying \ref{inequality} to each element within the $\max\{...\}$ and then selecting the largest term.
\[
\begin{array}{lll}
  |a_m^n| & \leq & \frac{C}{(m-N)(m-N-1)}\max\{|a_{m-2}^{n-1},...,|a^{n-1}_{m-(2N+2)}|\} \\
   & \leq & \frac{C^2}{(m-N)(m-N-1)(m-(2N+2)-N)(m-(2N+2)-N-1)}\max\{|a_{m-4}^{n-2},...,|a^{n-2}_{m-(2N+2)}|\} \\
   & \hdots &  \\
   & \leq & \frac{C^d}{(m-N)\tilde{!}}\max\{|a_{m-2d}^{n-d},...,|a^{n-d}_{m-(2N+2)d}|\}\\
   &\leq & B\frac{C^d}{(m-N)\tilde{!}}\\
   &\leq & B\frac{C^{\frac{m}{2N+2}}}{(m-N)\tilde{!}}.
\end{array}
\]
\end{proof}

\begin{lemma}\label{comp2}
The series
\[
\sum_{m=2N+2}^{\infty}\frac{C^{\frac{m}{2N+2} }}{(m-N)\tilde{!}}
\]
converges.
\end{lemma}

\begin{proof}
  Without loss of generality we consider
  \[
  \sum_{k=1}^{\infty}\frac{D^k}{k\tilde{!}},
  \]
  where $D$ is a positive constant.
  Consider
  \[
  \begin{array}{lll}
\frac{\frac{D^{k+1}}{(k+1)\tilde{!}}}{\frac{D^k}{k\tilde{!}}}&=&D\frac{k}{k+1}\frac{k-1}{k}\frac{k-3}{k-2}\frac{k-4}{k-3}\cdots \\
&\leq & D\frac{1}{k+1}\frac{k}{k}\frac{k-1}{k-2}\frac{k-3}{k-3}\frac{k-4}{k-5}\cdots\\
&\leq & D\frac{1}{k+1}.
\end{array}
  \]
 Hence we achieve convergence via the ratio test.

 We can now write
 \[
 \phi_n=\sum_{m=0}^{2n}a_m^nr^m+\sum_{m=2n}^{\infty}a_m^nr^m.
 \]
 The finite sum on the right is a partial sum of the Airy function solution to the test problem. This partial sum converges uniformly to the solution on any interval $[-R,R]$. The {\em tail} sum converges uniformly to $0$ on any $[-R,R]$ by a comparison argument using Lemmas \ref{comp1} and \ref{comp2}.
\end{proof}

\section{Concluding remarks}

We suggest several questions for future research. First, it should be straight forward to handle more general initial conditions via the principle of superposition. Then a nonhomogeneous term in the PDE can be considered. Second a more general class of operators, $Lu_{rr}-u_{tt}+V(r)u$ can be considered using similar techniques. Third, explore a more abstract situation where a Lagrange multiplier $\lambda$ is replaced with an approximation $\lambda_N$. Under what general conditions can convergence still be proved? Fourth, if $(\lambda_N)$ is a sequence of approximations for $\lambda$ such that $\lambda_N\rightarrow \lambda$, in some appropriate sense, then show that larger $N$ leads to faster convergence. Computational evidence for the last assertion was given in \cite{SiR}.

\end{document}